\documentclass[12pt,twoside]{article}
\usepackage{amsmath,amsthm,amssymb,amscd,ascmac,amsfonts}
\usepackage{mathrsfs}
\usepackage{stmaryrd}
\usepackage{braket}
\usepackage{accents}

\numberwithin{equation}{section}
\usepackage[dvips]{graphicx,color,psfrag}

\makeatletter
\newcommand{\figcaption}[1]{\def\@captype{figure}\caption{#1}}
\newcommand{\tblcaption}[1]{\def\@captype{table}\caption{#1}}
\makeatother

\newcommand{\Li}{\mathrm{Li}}

\newcommand{\bangle}{\atopwithdelims \langle \rangle}

\makeatletter

\ifcase\@ptsize
\def\rpkern{\mathchoice{\kern-1.45em}{\kern-1.11em}{}{}}%
\def\grpkern{\mathchoice{\kern-1.013em}{\kern-0.825em}{}{}}%
\or
\def\rpkern{\mathchoice{\kern-1.44em}{\kern-1.11em}{}{}}%
\def\grpkern{\mathchoice{\kern-1.00em}{\kern-0.81em}{}{}}%
\or
\def\rpkern{\mathchoice{\kern-1.472em}{\kern-1.14em}{}{}}%
\def\grpkern{\mathchoice{\kern-1.00em}{\kern-0.815em}{}{}}%
\fi

\def\minibullet{\mathchoice%
{\raise0.2ex\hbox{$\scriptstyle\bullet$}}%
{\raise0.26ex\hbox{$\scriptscriptstyle\bullet$}}{}{}}
\def\butabullet{\mathchoice%
{\raise0.8ex\hbox{$\scriptstyle\bullet$}{\kern-0.365em}%
\lower0.4ex\hbox{$\scriptstyle\bullet$}}%
{\raise0.75ex\hbox{$\scriptscriptstyle\bullet$}{\kern-0.335em}%
\lower0.25ex\hbox{$\scriptscriptstyle\bullet$}}{}{}}

\def\customprod#1#2%
{\operatorname{\coprod\kern#2{#1}\kern#2\prod}\displaylimits}

\renewcommand{\Re}{\mathrm{Re}\,}
\renewcommand{\Im}{\mathrm{Im}\,}

\newtheorem*{multitheorem}{\variable@name}

\theoremstyle{definition}
\newcommand{\variable@name}{Theorem}
\newtheorem*{multiproclaim}{\variable@name}

\theoremstyle{plain}
\newtheorem{thm}{Theorem}[section]
\newtheorem{prop}[thm]{Proposition}
\newtheorem{lem}[thm]{Lemma}
\newtheorem{cor}[thm]{Corollary}

\theoremstyle{definition}
\newtheorem{dfn}[thm]{Definition}

\begin{document}
\title{Bilateral zeta functions associated with the multiple sine functions}
\author{Genki Shibukawa}
\date{\empty}
\pagestyle{plain}

\maketitle


\begin{abstract}
We introduce two types bilateral zeta functions, which are related to the primitive and normalized multiple sine functions respectively. 
Further, we establish their main properties, that is, Fourier expansions, analytic continuations, differential and difference equations, special values. 
By applying these results, we obtain not only some generalization of the primitive and normalized multiple sine functions but also simple construction of the  multiple sine function theory. 
\end{abstract}

\section{Introduction}
The primitive multiple sine functions $\mathscr{S}_{r}(z)$  $(r \in \mathbb{Z}_{\geq 1})$
\begin{equation}
\mathscr{S}_{r}(z):=\exp\left(\frac{z^{r-1}}{r-1}\right)\prod_{m=1}^{\infty}\left(P_{r}\left(\frac{z}{m}\right)P_{r}\left(-\frac{z}{m}\right)^{(-1)^{r-1}}\right)^{m^{r-1}}
=\exp\left(\int_{0}^{z}\pi t^{r-1}\cot{\pi{t}}\,dt\right) \nonumber
\end{equation}
were introduced and studied by H\"{o}lder ($r=2$) and Kurokawa ($r\geq 3$) (see \cite{KK}). 
Here, we put 
$$
P_{r}(u):=(1-u)\exp\left(\sum_{j=1}^{r}\frac{u^{j}}{j}\right).
$$

Moreover, Kurokawa has constructed a normalized multiple sine functions $S_{r}(z)$.  
This function is defined as 
$$
S_{r}(z):=\exp\left(-\frac{\partial \zeta_{r}}{\partial s}(0,z)+(-1)^{r}\frac{\partial \zeta_{r}}{\partial s}(0,r-z)\right). 
$$
Here, $\zeta_{r}(s,z)$ is the multiple Hurwitz function, which extends holomorphic function on the whole $s$-plane except for possible simple poles at $s=1,\ldots,r$ under $\Re{z}>0$. 
$$
\zeta_{r}(s,z):=\sum_{m_{1},\ldots,m_{r}\geq 0}\frac{1}{(m_{1}+\cdots+m_{r}+z)^{s}}=\frac{1}{(r-1)!}\sum_{m\geq 0}\frac{(m+1)_{r-1}}{(m+z)^{s}}\,\,\,\,(\Re{z}>0). 
$$
The intimate relation between these two kinds of multiple sine functions is established by Kurokawa (see \cite{KK}).

Furthermore, for the normalized multiple sine functions $S_{r}(z)$, more generalization have been studied in Kurokawa-Wakayama\,\cite{KW}. 
Actually, for any $n \in \mathbb{Z}_{\geq 1}$, a generalization of the normalized multiple sine functions $S_{r,n}(z)$ is given by 
$$
S_{r,n}(z):=\exp\left(-\frac{\partial \zeta_{r}}{\partial s}(1-n,z)+(-1)^{r+n-1}\frac{\partial \zeta_{r}}{\partial s}(1-n,r-z)\right).
$$
We remark $S_{r,1}(z)=S_{r}(z)$. 
On the other hand, for the primitive type, such a kind of the generalization has not been studied yet.

In this article, we introduce two types bilateral zeta functions 
$$
H_{r}(s,z):=\sum_{m \in \mathbb{Z}}\frac{m^{r-1}}{(m+z)^{s}} 
$$
and 
$$
K_{r}(s,z):=\frac{1}{(r-1)!}\sum_{m \in \mathbb{Z}}\frac{(m+1)_{r-1}}{(m+z)^{s}}. 
$$
Since the derivatives at $s=0$ of $H_{r}(s,z)$ and $K_{r}(s,z)$ are equal to the primitive and normalized multiple sine functions up to some exponential factors,  
we define the generalized multiple sine functions by a derivation of $H_{r}(s,z)$ and $K_{r}(s,z)$ at $s=0,-1,-2, \ldots$. 
From this point of view, we are not only succeed in providing a generalization of the primitive multiple sine function, but also in providing simple construction of the multiple sine function theory.

In Section\,2, we give their fundamental properties, that is, Fourier expansions, analytic continuations, differential and difference equations, special values. 
In Section\,3, by applying results of Section\,2, we introduce a generalization of the primitive and normalized multiple sine functions and establish their basic properties. 



\section{Definition and basic properties}
Throughout the paper, we denote the ring of rational integers by $\mathbb{Z}$, 
the field of real numbers by $\mathbb{R}$, the field of complex numbers by $\mathbb{C}$ and $i:=\sqrt{-1}$. 
Further, we fix the branch of any complex numbers $c$. 
\begin{equation}
-\pi< \arg{c} \leq\pi. \nonumber
\end{equation}
In particular, we define $\arg{0}:=0$. 
We assume $r \in \mathbb{Z}_{\geq 2},n \in \mathbb{Z}_{\geq 1}, s \in \mathbb{C}$ and $z \in \mathbb{C}$ with $\Im{z}\geq 0$ unless otherwise specified. 
\begin{dfn}
Let $\Re{s}>r$. 
\\
\noindent
{\rm{(1)}}\,
We define a bilateral zeta function associated with the primitive multiple sine functions by the series
\begin{equation}
H_{r}(s,z):=\sum_{m \in \mathbb{Z}}\frac{m^{r-1}}{(m+z)^{s}}.
\end{equation}
{\rm{(2)}}\,We also introduce a bilateral zeta function associated with the normalized multiple sine functions as 
\begin{equation}
K_{r}(s,z):=\frac{1}{(r-1)!}\sum_{m \in \mathbb{Z}}\frac{(m+1)_{r-1}}{(m+z)^{s}}.
\end{equation}
Here, $(X)_{r}$ is the shifted factorial defined by 
$$
(X)_{r} := \begin{cases}
    1 & (r=0) \\
    X(X+1) \cdot (X+r-1) & ({\rm{otherwise}})
  \end{cases}.
$$
In particular, we put
\begin{equation}
\xi(s,z):=H_{1}(s,z)=K_{1}(s,z).
\end{equation}
\end{dfn}
These series absolutely converge for $\Re{s}>r$. 
Moreover, from the following well-known formula, $\xi(s,z)$ is continued analytically to $\mathbb{C}$ as a holomorphic function in $s$. 
\begin{lem}
\label{prop:Lipschitz formula}
{\rm{(1)}}\,
If $1>\Re{z}>0, \Im{z}\geq 0$, then 
\begin{equation}
\xi(s,z) =\zeta_{1}(s,z)+e^{-\pi{i}s}\zeta_{1}(s,1-z).
\end{equation}
By analytic continuation of the Hurwitz zeta function $\zeta_{1}(s,z)$, $\xi(s,z)$ extends analytically as a holomorphic function on the whole $s$-plane under the conditions $1>\Re{z}>0, \Im{z}\geq 0$.

\noindent
{\rm{(2)}}\,(See \cite{AAR}\,Chapter\,II.\,Exercise\,37) 
If $\Im{z}>0$, then 
\begin{align}
\xi(s,z) = \frac{(2\pi)^s}{\Gamma(s)}e^{-\frac{\pi}{2}is}\Li_{1-s}(e^{2{\pi}iz}).
\end{align}
Here, $\mathrm{Li}_{\alpha}(X)$ is the polylogarithm 
$$
\Li_{\alpha}(X):=\sum_{m=1}^{\infty}\frac{X^{m}}{m^{\alpha}}.
$$
From this expression, $\xi(s,z)$ extends analytically as a holomorphic function on the whole $s$-plane under the conditions $\Im{z}> 0$.
\end{lem}

By this fundamental formula, we also more results for $\xi(s,z)$, which are needed in later. 
\begin{cor}
\label{prop:special value of xi}
{\rm{(1)}}\,For any $s \in \mathbb{C}$, we have
\begin{equation}
\label{eq:period of xi}
\xi(s,z+1)=\xi(s,z).
\end{equation}
{\rm{(2)}}\,
\begin{align}
\label{eq:special value of xi(1)}
\xi(1,z)&=\pi{i}+\pi\cot{\pi z}, \\
\label{eq:special value of xi(N+1)1}
\xi(n+1,z)&=\frac{\pi^{n+1}}{n!}\sum_{k=0}^{n-1}{n \bangle k}(\cot{{\pi}z}-i)^{k+1}(\cot{{\pi}z}+i)^{n-k} \\
\label{eq:special value of xi(N+1)2}
&=\frac{1}{n!}\sum_{k=0}^{n-1}\sum_{l=0}^{k+1}{n \bangle k}\binom{k+1}{l}(-2{\pi}i)^{l}\xi(1,z)^{n+1-l}, \\
\label{eq:special value of xi(-N)}
\xi(1-n,z)&=0.
\end{align}
Here, ${n \bangle k}$ are Eulerian numbers which are the number of permutations of the set $\{1,\ldots,n\}$ having $k$ permutation ascents. 

\noindent
{\rm{(3)}}\,For any $n \in \mathbb{Z}_{\geq 1}$, 
\begin{equation}
\label{eq:special value of deriv for xi(-N)}
\frac{\partial \xi}{\partial s}(1-n,z)
=\frac{(n-1)!}{(2{\pi}i)^{n-1}}\Li_{n}(e^{2{\pi}iz}).
\end{equation}

\end{cor}
\begin{proof}
With the exception of (\ref{eq:special value of xi(N+1)1}) and (\ref{eq:special value of xi(N+1)2}), all properties are prove immediately from Lemma\,\ref{prop:Lipschitz formula}. 
(\ref{eq:special value of xi(N+1)1}) follows from the theorem of G.\,Rzadkowski\,\cite{R} that is if there exist some constants $a,\alpha,\beta$ such that 
$$
\frac{d}{dz}y(z)=a(y(z)-\alpha)(y(z)-\beta),
$$ 
then 
$$
y^{(n)}(z)=a^{n}\sum_{k=0}^{n-1}{n \bangle k}(y(z)-\alpha)^{k+1}(y(z)-\beta)^{n-k}.
$$
Actually, we remark
$$
\frac{d}{dz}\xi(1,z)=\frac{d}{dz}\pi \cot{\pi z}=-\pi^{2}(\cot{{\pi}z}-i)(\cot{{\pi}z}+i)
$$
and 
$$
\frac{d}{dz}\xi(n,z)=-n\xi(n+1,z). 
$$
Therefore, 
$$
\xi(n+1,z)=\frac{(-1)^{n}}{n!}\frac{d^{n}}{d^{n}z}\pi \cot{\pi z}=\frac{\pi^{n+1}}{n!}\sum_{k=0}^{n-1}{n \bangle k}(\cot{{\pi}z}-i)^{k+1}(\cot{{\pi}z}+i)^{n-k}.
$$
(\ref{eq:special value of xi(N+1)2}) follows from (\ref{eq:special value of xi(1)}) and (\ref{eq:special value of xi(N+1)1}). 
\end{proof}

The relation between $H_{r}(s,z)$ and $K_{r}(s,z)$ is given by the following Proposition.
\begin{prop}
\label{prop:relation between H and K}
{\rm{(1)}}\,Let ${r\brace k}$ be the Stirling numbers of the second kind. 
We have
\begin{equation}
\label{eq:relation H}
H_{r}(s,z)=\sum_{k=1}^{r}(-1)^{r-k}(k-1)!{r\brace k}K_{k}(s,z).
\end{equation}
{\rm{(2)}}\,Let ${r\brack k}$ be the Stirling numbers of the first kind. We have
\begin{equation}
\label{eq:relation K}
K_{r}(s,z)=\frac{1}{(r-1)!}\sum_{k=1}^{r}{r\brack k}H_{k}(s,z).
\end{equation}
\end{prop}
\begin{proof}
We recall the properties of the Stirling numbers
\begin{align}
m^{r}&=\sum_{k=1}^{r}(-1)^{r-k}{r\brace k}(m)_{k}=m\sum_{k=1}^{r}(-1)^{r-k}{r\brace k}(m+1)_{k-1}, \nonumber \\
(m)_{r}&=m(m+1)_{r-1}=\sum_{k=1}^{r}{r\brack k}m^{k}=m\sum_{k=1}^{r}{r\brack k}m^{k-1}. \nonumber
\end{align}
From these properties, we have
\begin{align}
H_{r}(s,z)&=\sum_{m \in \mathbb{Z}}\frac{1}{(m+z)^{s}}\sum_{k=1}^{r}(-1)^{r-k}{r\brace k}(m+1)_{k-1}
=\sum_{k=1}^{r}(-1)^{r-k}(k-1)!{r\brace k}K_{k}(s,z), \nonumber \\
K_{r}(s,z)&=\frac{1}{(r-1)!}\sum_{m \in \mathbb{Z}}\frac{1}{(m+z)^{s}}\sum_{k=1}^{r}{r\brack k}m^{k-1}
=\frac{1}{(r-1)!}\sum_{k=1}^{r}{r\brack k}H_{k}(s,z). \nonumber
\end{align}
\end{proof}

\begin{prop}[$\xi(s,z)$ expressions]
\label{prop:xi expression of H and K}
\noindent
{\rm{(1)}}\,
\begin{equation}
\label{eq:xi expression of H}
H_{r}(s,z)=\sum_{k=0}^{r-1}\binom{r-1}{k}(-z)^{r-1-k}\xi(s-k,z).
\end{equation}
\noindent
{\rm{(2)}}\,
\begin{equation}
\label{eq:xi expression of K}
K_{r}(s,z)=\frac{1}{(r-1)!}\sum_{k=1}^{r}{r\brack k}\sum_{l=0}^{k-1}\binom{k-1}{l}(-z)^{k-l-1}\xi(s-l,z).
\end{equation}
\end{prop}
\begin{proof}
{\rm{(1)}}\,By the binomial theorem, 
\begin{align}
H_{r}(s,z)&=\sum_{m \in \mathbb{Z}}\frac{(m+z-z)^{r-1}}{(m+z)^{s}} \nonumber \\
&=\sum_{m \in \mathbb{Z}}\frac{1}{(m+z)^{s}}\sum_{k=0}^{r-1}\binom{r-1}{k}(-z)^{r-1-k}(m+z)^{k} \nonumber \\
&=\sum_{k=0}^{r-1}\binom{r-1}{k}(-z)^{r-1-k}\xi(s-k,z). \nonumber
\end{align}
\noindent
{\rm{(2)}}\,This result follows from (\ref{eq:relation K}) and (\ref{eq:xi expression of H}) immediately. 
\end{proof}

From Lemma\,\ref{prop:Lipschitz formula} and Proposition\,\ref{prop:xi expression of H and K}, 
the functions $H_{r}(s,z)$ and $K_{r}(s,z)$ are analytically continued to $s \in \mathbb{C}$ as holomorphic functions. 
Therefore, although all results in this section hold for all $s$, we only need to prove them for large enough $\Re{s}$.

Further, from Corollary\,\ref{prop:special value of xi} and Proposition\,\ref{prop:xi expression of H and K}, we obtain some special values of $H_{r}(s,z)$ and $K_{r}(s,z)$. 
\begin{cor}[special values]
\label{prop:special values of H and K}
Let $n\geq 1$ and $r-1\geq p\geq 1$. 

\noindent
{\rm{(1)}}\,
\begin{align}
H_{r}(1,z)&=(-z)^{r-1}(\pi{i}+\pi{\cot{\pi z}}), \\
H_{r}(1+p,z)&=\binom{r-1}{p}(-z)^{r-p-1}(\pi{i}+\pi{\cot{\pi z}})+\sum_{k=0}^{p-1}\binom{r-1}{k}(-z)^{r-1-k}\frac{\pi^{p+1-k}}{(p-k)!} \nonumber \\
& \quad \cdot
\sum_{l=0}^{p-k-1}{p-k \bangle l}(\cot{\pi{z}}-i)^{l+1}(\cot{\pi{z}}+i)^{p-k-l}, \\
H_{r}(r+n,z)&=\sum_{k=0}^{r-1}\binom{r-1}{k}(-z)^{r-1-k}\frac{\pi^{r+n-k}}{(r+n-k-1)!} \nonumber \\
& \quad \cdot \sum_{l=0}^{r+n-k-2}{r+n-k-1 \bangle l}(\cot{\pi{z}}-i)^{l+1}(\cot{\pi{z}}+i)^{r+n-k-1-l},\\
H_{r}(1-n,z)&=0.
\end{align}
{\rm{(2)}}\,
\begin{align}
K_{r}(1,z)&=\frac{(1-z)_{r-1}}{(r-1)!}(\pi{i}+\pi{\cot{\pi z}}), \\
K_{r}(1+p,z)&=\sum_{k=1}^{p}{r\brack k}\sum_{l=0}^{k-1}\binom{k-1}{l}(-z)^{k-1-l}\frac{\pi^{1+p-l}}{(p-l)!} \nonumber \\
& \quad \cdot \sum_{q=0}^{p-l-1}{p-l \bangle q}(\cot{\pi{z}}-i)^{q+1}(\cot{\pi{z}}+i)^{p-l-q} \nonumber \\
& \quad +\frac{1}{(r-1)!}\sum_{k=p+1}^{r}{r\brack k}\sum_{l=0}^{k-1}\binom{k-1}{p}(-z)^{k-p-1}({\pi}i+{\pi}\cot{\pi{z}}) \nonumber \\
& \quad +\frac{1}{(r-1)!}\sum_{k=p+1}^{r}{r\brack k}\sum_{l=0}^{p-1}\binom{k-1}{l}(-z)^{k-l-1}\frac{\pi^{p+1-l}}{(p-l)!} \nonumber \\
& \quad \cdot \sum_{q=0}^{p-l-1}{p-l \bangle q}(\cot{\pi{z}}-i)^{q+1}(\cot{\pi{z}}+i)^{p-l-q}, \\
K_{r}(r+n,z)&=\frac{1}{(r-1)!}\sum_{k=1}^{r}{r\brack k}
\sum_{l=0}^{k-1}\binom{k-1}{l}(-z)^{k-1-l}\frac{\pi^{r+n-l}}{(r+n-l-1)!} \nonumber \\
& \quad \cdot \sum_{q=0}^{r+n-l-2}{r+n-l-1 \bangle q}(\cot{\pi{z}}-i)^{q+1}(\cot{\pi{z}}+i)^{r+n-l-1-q},\\
K_{r}(1-n,z)&=0.
\end{align}
\end{cor}

We also have some expression of $H_{r}(s,z)$ and $K_{r}(s,z)$ by using the multiple Hurwitz zeta functions. 
\begin{prop}[$\zeta_{r}(s,z)$ expressions]
\label{prop:zata expression of H and K}
{\rm{(1)}}\,If $1>\Re{z}>0$, $\Im{z}\geq 0$, then for any $s \in \mathbb{C}$, we have
\begin{equation}
\label{eq:zeta expression of H}
H_{r}(s,z)=\sum_{k=1}^{r}(-1)^{r-k}(k-1)!{r\brace k}(\zeta_{k}(s,z)+(-1)^{k-1}e^{-\pi{i}s}\zeta_{k}(s,k-z)).
\end{equation}
{\rm{(2)}}\,If $r>\Re{z}>0$, $\Im{z}\geq 0$, then for any $s \in \mathbb{C}$, we have
\begin{align}
\label{eq:zeta expression of K}
K_{r}(s,z)&=\zeta_{r}(s,z)+(-1)^{r-1}e^{-\pi{i}s}\zeta_{r}(s,r-z) \\
\label{eq:zeta expression of K2}
&=z^{-s}+\zeta_{r}(s,1+z)+(-1)^{r-1}e^{-\pi{i}s}\zeta_{r}(s,r-z).
\end{align}
{\rm{(3)}}\,If $1>\Re{z}>-1$, $\Im{z}\geq 0$, then for any $s \in \mathbb{C}$, we have
\begin{equation}
\label{eq:zeta expression of H2}
H_{r}(s,z)=\sum_{k=0}^{r-1}\binom{r-1}{k}(-z)^{r-1-k}(\zeta(s-k,1+z)+(-1)^{k-1}e^{-\pi{i}s}\zeta(s-k,1-z)).
\end{equation}
\end{prop}
\begin{proof}

\noindent
{\rm{(1)}}\,
(\ref{eq:zeta expression of H}) follows from (\ref{eq:zeta expression of K}) and (\ref{eq:relation H}).


\noindent
{\rm{(2)}}\,
We decompose the sum $K_{r}(s,z)$ as follows. 
$$
K_{r}(s,z)=\frac{1}{(r-1)!}\sum_{m\geq 0}\frac{(m+1)_{r-1}}{(m+z)^{s}}+\frac{1}{(r-1)!}\sum_{m\geq 0}\frac{(-m)_{r-1}}{(-m-1+z)^{s}}.
$$
Here, we remark $0<\arg{(-m-1+z)}<\pi$ and 
$$
\sum_{m\geq 0}\frac{(-m)_{r-1}}{(-m-1+z)^{s}}
=\sum_{m\geq r-1}\frac{(-m)_{r-1}}{(-m-1+z)^{s}}
=(-1)^{r-1}e^{-\pi{i}s}\sum_{m\geq 0}\frac{(m+1)_{r-1}}{(m+r-z)^{s}}. 
$$
From this calculation, we have (\ref{eq:zeta expression of K}). 
(\ref{eq:zeta expression of K2}) follows from (\ref{eq:zeta expression of K}) and the definition of $\zeta_{r}(s,z)$.

\noindent
{\rm{(3)}}\,
Since $r\geq 2$ and $0<\arg{(-m-1+z)}<\pi$ under the condition, 
\begin{align}
H_{r}(s,z)&=\sum_{m\not=0}\frac{m^{r-1}}{(m+z)^{s}} \nonumber \\
&=\sum_{m\geq 0}\frac{(m+1+z-z)^{r-1}}{(m+1+z)^{s}}+(-1)^{r-1}e^{-\pi{i}s}\sum_{m\geq 0}\frac{(m+1-z+z)^{r-1}}{(m+1-z)^{s}} \nonumber \\
&=\sum_{k=0}^{r-1}\binom{r-1}{k}(-z)^{r-1-k}\zeta(s-k,1+z) \nonumber \\
& \quad +(-1)^{r-1}e^{-\pi{i}s}\sum_{k=0}^{r-1}\binom{r-1}{k}z^{r-1-k}\zeta(s-k,1-z) \nonumber \\
&=\sum_{k=0}^{r-1}\binom{r-1}{k}(-z)^{r-1-k}(\zeta(s-k,1+z)+(-1)^{k-1}e^{-\pi{i}s}\zeta(s-k,1-z)). \nonumber
\end{align}
\end{proof}


In addition, by applying the Fourier expansion for $\xi(s,z)$, we obtain the following Fourier expansions of $H$ and $K$. 
\begin{prop}[Fourier expansions]
\label{prop:Fourier expansions of H and K}
Let $\Im{z}>0$.

\noindent
{\rm{(1)}}\,For any $s \in \mathbb{C}$, 
\begin{equation}
\label{eq:Fourier expansions of H}
H_{r}(s,z)=\frac{(2\pi)^s}{\Gamma(s)}e^{-\frac{\pi}{2}is}(-z)^{r-1}
\sum_{k=0}^{r-1}\frac{(1-r)_{k}(1-s)_{k}}{k!}(2\pi{i}z)^{-k}\Li_{1+k-s}(e^{2{\pi}iz}).
\end{equation}
\noindent
{\rm{(2)}}\,
\begin{equation}
\label{eq:Fourier expansions of K}
K_{r}(s,z)=\frac{(2\pi)^s}{\Gamma(s)}\frac{e^{-\frac{\pi}{2}is}}{(r-1)!}\sum_{k=1}^{r}{r\brack k}(-z)^{k-1}
\sum_{l=0}^{k-1}\frac{(1-k)_{l}(1-s)_{l}}{l!}(2\pi{i}z)^{-l}\Li_{1+l-s}(e^{2{\pi}iz}).
\end{equation}
\end{prop}
Applying Proposition\,\ref{prop:Fourier expansions of H and K}, we obtain special values of derivation of $H$ and $K$. 
\begin{prop}
\label{prop:special values of H and K 2}


\noindent
{\rm{(1)}}\,If $\Im{z}> 0$ or $1>z>-1$, then for any $n \in \mathbb{Z}_{\geq 1}$ and $s \in \mathbb{C}$, 
\begin{equation}
\label{eq:special values of deriv of H}
\frac{\partial H_{r}}{\partial s}(1-n,z)=\sum_{k=0}^{r-1}\frac{(n+k-1)!}{(2\pi{i})^{n+k-1}}\binom{r-1}{k}(-z)^{r-1-k}
\Li_{n+k}(e^{2\pi{i}z}).
\end{equation}

\noindent
{\rm{(2)}}\,If $\Im{z}> 0$ or $r>z> 0$, for any $n \in \mathbb{Z}_{\geq 1}$ and $s \in \mathbb{C}$, 
\begin{equation}
\label{eq:special values of deriv of K}
\frac{\partial K_{r}}{\partial s}(1-n,z)=\frac{1}{(r-1)!}\sum_{k=1}^{r}{r\brack k}\sum_{l=0}^{k-1}\binom{k-1}{l}\frac{(n+l-1)!}{(2\pi{i})^{n+l-1}}(-z)^{k-l-1}
\Li_{n+l}(e^{2\pi{i}z}).
\end{equation}
Moreover, if $n>1$, this result holds for $\Im{z}> 0$ or $r>z \geq 0$. 
\end{prop}
\begin{proof}
If $\Im{z}> 0$, then (\ref{eq:special values of deriv of H}) and (\ref{eq:special values of deriv of K}) follow from Proposition\,\ref{prop:Fourier expansions of H and K} immediately. 
Further, we remark the right hand sides of (\ref{eq:special values of deriv of H}) and (\ref{eq:special values of deriv of K}) is well-defined under the conditions. 
Thus, it is enough to show the well-definedness of the left hand sides of (\ref{eq:special values of deriv of H}) and (\ref{eq:special values of deriv of K}).  

\noindent
{\rm{(1)}}\,By (\ref{eq:zeta expression of H2}), $H_{r}(s,z)$ is holomorphic function of $s$ under $\Im{z}\geq  0$, $1>\Re{z}>-1$. 
Hence, by identity theorem, we have the conclusion.

\noindent
{\rm{(2)}}\,By (\ref{eq:zeta expression of K}), $K_{r}(s,z)$ is an entire function for $s$ under $\Im{z}\geq  0$, $r>\Re{z}>0$. 
Thus, by identity theorem, we have (\ref{eq:special values of deriv of K}) under $\Im{z}> 0$ or $r>z> 0$. 

On the other hands, by (\ref{eq:zeta expression of K2}), the left hand sides of (\ref{eq:special values of deriv of K}) has the following expression. 
\begin{align}
\frac{\partial H_{r}}{\partial s}(1-n,z)
&=-z^{n-1}\log{z}-\pi{i}(-1)^{r+n}\zeta_{r}(1-n,r-z) \nonumber \\
& \quad +\frac{\partial \zeta_{r}}{\partial s}(1-n,1+z)+\frac{\partial \zeta_{r}}{\partial s}(1-n,r-z). \nonumber  
\end{align}
When $n>1$, from the above expression of $\frac{\partial H_{r}}{\partial s}(1-n,z)$ is entire in $\Im{z}\geq  0$, $r>\Re{z}\geq 0$. 
Therefore, by identity theorem, we obtain the conclusion. 

\end{proof}
\begin{cor}
\label{prop:special values of deriv of H and K}
{\rm{(1)}}\,
\begin{align}
\frac{\partial H_{r}}{\partial s}(1-n,0)&=\frac{(n+r-2)!}{(2\pi{i})^{n+r-2}}\zeta(n+r-1), \\
\frac{\partial K_{r}}{\partial s}(-n,0)&=\frac{1}{(r-1)!}\sum_{k=1}^{r}{r\brack k}\frac{(n+k-1)!}{(2\pi{i})^{n+k-1}}\zeta(n+k).
\end{align}
Here, $\zeta(s):=\zeta_{1}(s,1)$ is the Riemann zeta function.

\noindent
{\rm{(2)}}\,
\begin{align}
\frac{\partial H_{r}}{\partial s}\left(1-n,\frac{1}{2}\right)
&=\sum_{k=0}^{r-1}\frac{(n+k-1)!}{(2\pi{i})^{n+k-1}}\binom{r-1}{k}(-2)^{-r+1+k}(2^{1-n-k}-1)\zeta(n+k), \\
\frac{\partial K_{r}}{\partial s}\left(1-n,\frac{1}{2}\right)
&=\frac{1}{(r-1)!}\sum_{k=1}^{r}{r\brack k}\sum_{l=0}^{k-1}\binom{k-1}{l}\frac{(n+l-1)!}{(2\pi{i})^{n+l-1}}(-2)^{-k+l+1}
(2^{1-n-l}-1)\zeta(n+l).
\end{align}
We remark $\lim_{n+k \to 1}(2^{1-n-k}-1)\zeta(n+k)=-\log{2}$.
\end{cor}

\begin{prop}[difference equations for $z$]
\label{prop:difference eq of H and K}
Let $l \in \mathbb{Z}_{\geq 1}$. 

\noindent
{\rm{(1)}}\,
\begin{equation}
\label{eq:difference eq of H}
H_{r}(s,z+l)=\sum_{k=0}^{r-1}(-l)^{k}\binom{r-1}{k}H_{r-k}(s,z).
\end{equation}
{\rm{(2)}}\,
\begin{equation}
\label{eq:difference eq of K}
K_{r}(s,z+l)=K_{r}(s,z)-\sum_{k=0}^{l-1}K_{r-1}(s,z+k).
\end{equation}
\end{prop}
\begin{proof}
{\rm{(1)}}\,
$$
H_{r}(s,z+l)
=\sum_{m \in \mathbb{Z}}\frac{m^{r-1}}{(m+z+l)^{s}} 
=\sum_{m \in \mathbb{Z}}\frac{(m-l)^{r-1}}{(m+z)^{s}} 
=\sum_{m \in \mathbb{Z}}\frac{1}{(m+z)^{s}}\sum_{k=0}^{r-1}(-l)^{k}\binom{r-1}{k}m^{r-k-1}.
$$
Thus, by the definition of $H_{r}(s,z)$, we have (\ref{eq:difference eq of H}). 

\noindent
{\rm{(2)}}\,First, we remark 
\begin{align}
K_{r}(s,z+1)
&=\frac{1}{(r-1)!}\sum_{m \in \mathbb{Z}}\frac{(m)_{r-1}}{(m+z)^{s}} \nonumber \\
&=\frac{1}{(r-1)!}\sum_{m \in \mathbb{Z}}\frac{(m+1)_{r}-(r-1)(m+1)_{r-1}}{(m+z)^{s}} \nonumber \\
&=K_{r}(s,z)-K_{r-1}(s,z). \nonumber
\end{align}
Thus, 
\begin{align}
K_{r}(s,z+l)&=K_{r}(s,z+l-1)-K_{r-1}(s,z+l-1) \nonumber \\
&=K_{r}(s,z)-\sum_{k=0}^{l-1}K_{r-1}(s,z+k). \nonumber
\end{align}
\end{proof}
\begin{prop}[difference equations for $s$]
\label{prop:difference eq 2 of H and K}

\noindent
{\rm{(1)}}\,
\begin{equation}
\label{eq:difference eq 2 of H}
H_{r}(s-l,z)=\sum_{p=0}^{l}\binom{l}{p}z^{l-p}H_{r+p}(s,z).
\end{equation}
{\rm{(2)}}\,
\begin{equation}
\label{eq:difference eq 2 of K}
K_{r}(s-l,z)=\frac{1}{(r-1)!}\sum_{k=1}^{r}{r \brack k}\sum_{p=0}^{l}\binom{l}{p}z^{l-p}
\sum_{q=1}^{k+p}(-1)^{k+p-q}(q-1)!{k+p \brace q}K_{q}(s,z). 
\end{equation}
\end{prop}
\begin{proof}
{\rm{(1)}}\,By the analytic continuation of $H_{r}(s,z)$, it is enough to show the assertion (\ref{eq:difference eq 2 of H}) when $\Re{s}>r+l$. 
In this case, it is easy to see that
\begin{align}
H_{r}(s-l,z)
&=\sum_{m \in \mathbb{Z}}\frac{m^{r-1}}{(m+z)^{s}}(m+z)^{l} \nonumber \\
&=\sum_{m \in \mathbb{Z}}\frac{m^{r-1}}{(m+z)^{s}}\sum_{p=0}^{l}\binom{l}{p}z^{l-p}m^{p} \nonumber \\
&=\sum_{p=0}^{l}\binom{l}{p}z^{l-p}\sum_{m \in \mathbb{Z}}\frac{m^{r+p-1}}{(m+z)^{s}} \nonumber \\
&=\sum_{p=0}^{l}\binom{l}{p}z^{l-p}H_{r+p}(s,z). \nonumber
\end{align}
{\rm{(2)}}\,From (\ref{eq:relation K}), (\ref{eq:difference eq 2 of H}) and (\ref{eq:relation H}), we have 
\begin{align}
K_{r}(s-l,z)
&=\frac{1}{(r-1)!}\sum_{k=1}^{r}{r\brack k}H_{k}(s-l,z) \nonumber \\
&=\frac{1}{(r-1)!}\sum_{k=1}^{r}{r\brack k}\sum_{p=0}^{l}\binom{l}{p}z^{l-p}H_{k+p}(s,z) \nonumber \\
&=\frac{1}{(r-1)!}\sum_{k=1}^{r}{r \brack k}\sum_{p=0}^{l}\binom{l}{p}z^{l-p}
\sum_{q=1}^{k+p}(-1)^{k+p-q}(q-1)!{k+p \brace q}K_{q}(s,z). \nonumber
\end{align} 
\end{proof}
\begin{prop}[multiplication formulas]
\label{prop:multiplication of H and K}
{\rm{(1)}}\,
\begin{equation}
\label{eq:multiplication of H}
H_{r}(s,Nz)=N^{-s}\sum_{k=0}^{N-1}\sum_{l=0}^{r-1}\binom{r-1}{l}k^{l}N^{r-l-1}H_{r-l}\left(s,z+\frac{k}{N}\right).
\end{equation}
{\rm{(2)}}\,
\begin{equation}
\label{eq:multiplication of K}
K_{r}(s,Nz)=N^{-s}\sum_{k_{1},\ldots,k_{r}\geq 0}^{N-1}K_{r}\left(s,z+\frac{k_{1}+\cdots+k_{r}}{N}\right).
\end{equation}
\end{prop}
\begin{proof}
{\rm{(1)}}\,
We decompose the sum of $n$ by setting $n=Nj+k$. 
\begin{align}
H_{r}(s,Nz)&=\sum_{j \in \mathbb{Z}}\sum_{k=0}^{N-1}\frac{(Nj+k)^{r-1}}{(Nj+k+Nz)^{s}} \nonumber \\
&=N^{-s+r-1}\sum_{k=0}^{N-1}\sum_{j \in \mathbb{Z}}\frac{1}{\left(j+z+\frac{k}{N}\right)^{s}}\sum_{l=0}^{r-1}\binom{r-1}{l}\left(\frac{k}{N}\right)^{l}j^{r-l-1} \nonumber \\
&=N^{-s}\sum_{k=0}^{N-1}\sum_{l=0}^{r-1}\binom{r-1}{l}k^{l}N^{r-l-1}H_{r-l}\left(s,z+\frac{k}{N}\right). \nonumber
\end{align}
{\rm{(2)}}\,First, we remark 
\begin{align}
\zeta_{r}(s,Nz)&=\sum_{j_{1},\ldots,j_{r} \in \mathbb{Z}}\sum_{0\leq k_{1},\ldots,k_{r}\leq N-1}
\frac{1}{((Nj_{1}+k_{1})+\ldots+(Nj_{r}+k_{r})+Nz)^{s}} \nonumber \\
&=N^{-s}\sum_{j_{1},\ldots,j_{r} \in \mathbb{Z}}\sum_{0\leq k_{1},\ldots,k_{r}\leq N-1}
\frac{1}{\left(j_{1}+\ldots+j_{r}+z+\frac{k_{1}+\ldots+k_{r}}{N}\right)^{s}} \nonumber \\
&=N^{-s}\sum_{0\leq k_{1},\ldots,k_{r}\leq N-1}\zeta_{r}\left(s,z+\frac{k_{1}+\ldots+k_{r}}{N}\right) \nonumber
\end{align}
and 
\begin{align}
\zeta_{r}(s,r-Nz)&=\sum_{j_{1},\ldots,j_{r} \in \mathbb{Z}}\sum_{1\leq k_{1},\ldots,k_{r}\leq N}
\frac{1}{((Nj_{1}+k_{1})+\ldots+(Nj_{r}+k_{r})-Nz)^{s}} \nonumber \\
&=N^{-s}\sum_{j_{1},\ldots,j_{r} \in \mathbb{Z}}\sum_{0\leq k_{1},\ldots,k_{r}\leq N-1}
\frac{1}{\left(j_{1}+\ldots+j_{r}+r-\left(z+\frac{k_{1}+\ldots+k_{r}}{N}\right)\right)^{s}} \nonumber \\
&=N^{-s}\sum_{0\leq k_{1},\ldots,k_{r}\leq N-1}\zeta_{r}\left(s,r-\left(z+\frac{k_{1}+\ldots+k_{r}}{N}\right)\right). \nonumber
\end{align}
Hence, by using (\ref{eq:zeta expression of K}), 
\begin{align}
K_{r}(s,Nz)&=\zeta_{r}(s,Nz)+(-1)^{r-1}e^{-\pi{i}s}\zeta_{r}(s,r-Nz) \nonumber \\
&=N^{-s}\sum_{0\leq k_{1},\ldots,k_{r}\leq N-1} \nonumber \\
& \quad \cdot \left\{\zeta_{r}\left(s,z+\frac{k_{1}+\ldots+k_{r}}{N}\right)+(-1)^{r-1}e^{-\pi{i}s}\zeta_{r}\left(s,r-\left(z+\frac{k_{1}+\ldots+k_{r}}{N}\right)\right)\right\} \nonumber \\
&=N^{-s}\sum_{k_{1},\ldots,k_{r}\geq 0}^{N-1}K_{r}\left(s,z+\frac{k_{1}+\cdots+k_{r}}{N}\right). \nonumber
\end{align}
\end{proof}
Moreover, derivation formulas for $H_{r}(s,z)$ and $K_{r}(s,z)$ are obtained by easy calculations. 
\begin{prop}
\label{prop:derivation of H and K}
\begin{align}
\frac{\partial }{\partial z}H_{r}(s,z)&=-sH_{r}(s+1,z), \\
\frac{\partial }{\partial z}K_{r}(s,z)&=-sK_{r}(s+1,z).
\end{align}
\end{prop}
As a corollary of Proposition\,\ref{prop:derivation of H and K}, we obtain the following important formulas, that is the logarithmic derivative of the generalized primitive and normalized multiple sine functions. 
\begin{cor}
\label{prop:deriv of H and K at s}
{\rm{(1)}}\,
\begin{align}
\label{eq:deriv of H s0}
\frac{\partial }{\partial z}\frac{\partial H_{r}}{\partial s}(0,z)&=-H_{r}(1,z)=-(-z)^{r-1}(\pi{i}+\pi{\cot{\pi z}}), \\
\label{eq:deriv of K s0}
\frac{\partial }{\partial z}\frac{\partial K_{r}}{\partial s}(0,z)&=-K_{r}(1,z)=-\frac{(1-z)_{r-1}}{(r-1)!}(\pi{i}+\pi{\cot{\pi z}}).
\end{align}
{\rm{(2)}}\,For any $n>1$, we have
\begin{align}
\label{eq:deriv of H s}
\frac{\partial }{\partial z}\frac{\partial H_{r}}{\partial s}(1-n,z)&=(n-1)\frac{\partial H_{r}}{\partial s}(2-n,z), \\
\label{eq:deriv of K s}
\frac{\partial }{\partial z}\frac{\partial K_{r}}{\partial s}(1-n,z)&=(n-1)\frac{\partial K_{r}}{\partial s}(2-n,z).
\end{align}
\end{cor}
\begin{proof}
Since the proofs for $H$ and $K$ are similar, we only give the proof for $H$.

\noindent
{\rm{(1)}}\,
$H_{r}(s,z)$ is holomorphic around $s=0$ and $s=1$. 
Hence we have expressions 
$$
H_{r}(s,z)=\frac{\partial H_{r}}{\partial s}(0,z)s+{O}(s^{2})
=H_{r}(1,z)+{O}(s-1).
$$
Thus, 
$$
\frac{\partial H_{r}}{\partial z}(s,z)=\frac{\partial }{\partial z}\frac{\partial H_{r}}{\partial s}(0,z)s+{O}(s^{2}).
$$
On the other hands, from Proposition\,\ref{prop:derivation of H and K}, 
$$
\frac{\partial H_{r}}{\partial z}(s,z)=-sH_{r}(s+1,z)=-H_{r}(1,z)s+{O}(s^{2}).
$$
Therefore, 
$$
\frac{\partial }{\partial z}\frac{\partial H_{r}}{\partial s}(0,z)=-H_{r}(1,z)=-(-z)^{r-1}(\pi{i}+\pi{\cot{\pi z}}).
$$

\noindent
{\rm{(2)}}\,For $n>1$, $H_{r}(s,z)$ is holomorphic around $s=1-n$ and $s=2-n$. 
Hence we have expressions 
$$
H_{r}(s,z)=\frac{\partial H_{r}}{\partial s}(1-n,z)(s-1+n)+{O}((s-1+n)^{2})
=\frac{\partial H_{r}}{\partial s}(2-n,z)(s-2+n)+{O}((s-2+n)^{2}).
$$
Thus, 
$$
\frac{\partial H_{r}}{\partial z}(s,z)=\frac{\partial }{\partial z}\frac{\partial H_{r}}{\partial s}(1-n,z)(s-1+n)+{O}((s-1+n)^{2})
$$
On the other hands, from Proposition\,\ref{prop:derivation of H and K}, 
\begin{align}
\frac{\partial H_{r}}{\partial z}(s,z)&=-sH_{r}(s+1,z) \nonumber \\
&=-(s-1+n+1-n)\left\{\frac{\partial H_{r}}{\partial s}(2-n,z)(s-1+n)+{O}((s-1+n)^{2})\right\} \nonumber \\
&=(n-1)\frac{\partial H_{r}}{\partial s}(2-n,z)(s-1+n)+{O}((s-1+n)^{2}). \nonumber
\end{align}
By comparing of $s-1+n$, we obtain the conclusion.
\end{proof}

\section{A generalization of the multiple sine functions}
\begin{dfn}
We define the generalized primitive and normalized multiple sine functions as 
\begin{align}
\mathscr{S}_{r,n}(z)&:=\exp\left(-\frac{\partial H_{r}}{\partial s}(1-n,z)\right)\,\,\,\,\,(\text{$\Im{z}>0$ or $1>z>-1$}), \\
\mathcal{S}_{r,n}(z)&:=\exp\left(-\frac{\partial K_{r}}{\partial s}(1-n,z)\right)\,\,\,\,\,(\text{$\Im{z}>0$ or $r>z>0$}).
\end{align}
If $n>1$, $\mathcal{S}_{r,n}(z)$ is defined on $\Im{z}> 0$ or $r>z \geq 0$. 
Unless otherwise stated, we assume the above conditions of $z$. 
\end{dfn}

From the definitions of $\mathscr{S}_{r,n}(z), \mathcal{S}_{r,n}(z)$ and the previous results, we obtain the following propositions immediately.

\begin{prop}
\label{prop:relation between mathscr S and mathcal S}
\begin{align}
\label{eq:relation mathcal S and mathscr S}
\mathscr{S}_{r,n}(z)&=\prod_{k=1}^{r}\mathcal{S}_{k,n}(z)^{(-1)^{r-k}(k-1)!{r\brace k}}, \\
\mathcal{S}_{r,n}(z)&=\prod_{k=1}^{r}\mathscr{S}_{k,n}(z)^{\frac{1}{(r-1)!}{r\brack k}}.
\end{align}
\end{prop}

\begin{prop}
\label{prop:Fourier expansion of mathcal S and mathscr S}
Let $n\geq 1$.  
\begin{align}
\mathscr{S}_{r,n}(z)&=\exp\left(-\sum_{k=0}^{r-1}\frac{(n+k-1)!}{(2\pi{i})^{n+k-1}}\binom{r-1}{k}(-z)^{r-1-k}
\Li_{n+k}(e^{2\pi{i}z})\right), \\
\mathcal{S}_{r,n}(z)&=\exp\left(-\frac{1}{(r-1)!}\sum_{k=1}^{r}{r\brack k}\sum_{l=0}^{k-1}\binom{k-1}{l}\frac{(n+l-1)!}{(2\pi{i})^{n+l-1}}(-z)^{k-l-1}
\Li_{n+l}(e^{2\pi{i}z})\right). 
\end{align}
\end{prop}
\begin{prop}
\label{prop:special values of mathscr S and mathcal S}
{\rm{(1)}}\,
\begin{align}
\mathscr{S}_{r,n}(0)&=\exp{\left(-\frac{(n+r-2)!}{(2\pi{i})^{n+r-2}}\zeta(n+r-1)\right)}, \\
\mathcal{S}_{r,n+1}(0)&=\prod_{k=1}^{r}\exp{\left(-\frac{1}{(r-1)!}{r\brack k}\frac{(n+k-1)!}{(2\pi{i})^{n+k-1}}\zeta(n+k)\right)}.
\end{align}

\noindent
{\rm{(2)}}\,
\begin{align}
\mathscr{S}_{r,n}\left(\frac{1}{2}\right)
&=\prod_{k=0}^{r-1}\exp{\left(-\frac{(n+k-1)!}{(2\pi{i})^{n+k-1}}\binom{r-1}{k}(-2)^{-r+1+k}(2^{1-n-k}-1)\zeta(n+k)\right)}, \\
\mathcal{S}_{r,n}\left(\frac{1}{2}\right)
&=\prod_{k=1}^{r}\prod_{l=0}^{k-1}\exp{\left(-\frac{1}{(r-1)!}{r\brack k}\binom{k-1}{l}\frac{(n+l-1)!}{(2\pi{i})^{n+l-1}}(-2)^{-k+l+1}
(2^{1-n-l}-1)\zeta(n+l)\right)}.
\end{align}
We remark $\lim_{n+k \to 1}(2^{1-n-k}-1)\zeta(n+k)=-\log{2}$.
\end{prop}
\begin{prop}
\label{prop:difference eq of mathcal S and mathscr S}
\begin{align}
\label{eq:difference eq of mathcal S and mathscr S}
\mathscr{S}_{r,n}(z+l)&=\prod_{k=0}^{r-1}\mathscr{S}_{r-k,n}(z)^{(-l)^{k}\binom{r-1}{k}}, \\
\mathcal{S}_{r,n}(z+l)&=\mathcal{S}_{r,n}(z)\prod_{k=0}^{l-1}\mathcal{S}_{r-1,n}(z+k)^{-1}.
\end{align}
\end{prop}
\begin{prop}
\label{prop:difference eq 2 of mathcal S and mathscr S}
\begin{align}
\label{eq:difference eq 2 of mathcal S and mathscr S}
\mathscr{S}_{r,n}(z)&=\prod_{p=0}^{n-1}\exp\left(\binom{n-1}{p}z^{n-p-1}\log{\mathscr{S}_{r+p,1}(z)}\right), \\
\mathcal{S}_{r,n}(z)&=\prod_{k=1}^{r}\prod_{p=0}^{n-1}\prod_{q=1}^{k+p}
\exp\left((-1)^{k+p-q}\frac{(q-1)!}{(r-1)!}{r \brack k}\binom{n-1}{p}{k+p \brace q}z^{n-1-p}\log{\mathcal{S}_{q,1}(z)}\right). 
\end{align}
\end{prop}
\begin{prop}
\label{prop:multiplication of mathcal S and mathscr S}
\begin{align}
\label{eq:multiplication of mathcal S and mathscr S}
\mathscr{S}_{r,n}(Nz)&=\prod_{k=0}^{N-1}\prod_{l=0}^{r-1}\mathscr{S}_{r-l,n}\left(z+\frac{k}{N}\right)^{\binom{r-1}{l}k^{l}N^{n+r-l-2}}, \\
\mathcal{S}_{r,n}(Nz)&=\prod_{k_{1},\ldots,k_{r}\geq 0}^{N-1}\mathcal{S}_{r}\left(z+\frac{k_{1}+\cdots+k_{r}}{N}\right)^{N^{n-1}}.
\end{align}
\end{prop}
\begin{prop}
\label{prop:log deriv of mathscr S and mathcal S}
{\rm{(1)}}\,
\begin{align}
\label{eq:deriv of mathscr S 1}
\frac{\mathscr{S}_{r,1}'(z)}{\mathscr{S}_{r,1}(z)}&=(-z)^{r-1}(\pi{i}+\pi{\cot{\pi z}}), \\
\label{eq:deriv of mathcal S 1}
\frac{\mathcal{S}_{r,1}'(z)}{\mathcal{S}_{r,1}(z)}&=\frac{(1-z)_{r-1}}{(r-1)!}(\pi{i}+\pi{\cot{\pi z}}).
\end{align}
{\rm{(2)}}\,For any $n \in \mathbb{Z}_{\geq 1}$, 
\begin{align}
\label{eq:deriv of mathscr S 2}
\frac{\mathscr{S}_{r,n+1}'(z)}{\mathscr{S}_{r,n+1}(z)}&=n\log{\mathscr{S}_{r,n}(z)}, \\
\label{eq:deriv of mathscr S 2}
\frac{\mathcal{S}_{r,n+1}'(z)}{\mathcal{S}_{r,n+1}(z)}&=n\log{\mathcal{S}_{r,n}(z)}.
\end{align}
Consequently we have
\begin{align}
\mathscr{S}_{r,n+1}(z)&=\exp{\left(-\frac{(n+r-1)!}{(2\pi{i})^{n+r-1}}\zeta(n+r)\right)}
\exp{\left(n\int_{0}^{z}\log{\mathscr{S}_{r,n}(t)}\,dt\right)}, \\
\mathcal{S}_{r,n+1}(z)&=\exp{\left(-\frac{1}{(r-1)!}\sum_{k=1}^{r}{r\brack k}\frac{(n+k-1)!}{(2\pi{i})^{n+k-1}}\zeta(n+k)\right)}
\exp{\left(n\int_{0}^{z}\log{\mathcal{S}_{r,n}(t)}\,dt\right)}.
\end{align}
\end{prop}
\begin{prop}
For any $l \in \mathbb{Z}_{\geq 1}$, 
\label{prop:Raabe formulas}
\begin{align}
\label{eq:Raabe formulas}
n\int_{0}^{l}\log{\mathscr{S}_{r,n}(z+t)}\,dt&=\sum_{k=1}^{r-1}(-l)^{k}\binom{r-1}{k}\log{\mathscr{S}_{r-k,n+1}(z)}, \\
\label{eq:Raabe formulas 2}
n\int_{0}^{l}\log{\mathcal{S}_{r,n}(z+t)}\,dt&=-\sum_{k=0}^{l-1}\log{\mathcal{S}_{r-1,n+1}(z+k)}.
\end{align}
\end{prop}
Actually, Proposition\,\ref{prop:relation between mathscr S and mathcal S}, 
\ref{prop:Fourier expansion of mathcal S and mathscr S}, 
\ref{prop:special values of mathscr S and mathcal S}, 
\ref{prop:difference eq of mathcal S and mathscr S}, 
\ref{prop:difference eq 2 of mathcal S and mathscr S}, 
\ref{prop:multiplication of mathcal S and mathscr S}, 
\ref{prop:log deriv of mathscr S and mathcal S} follow from 
Proposition\,\ref{prop:relation between H and K}, 
\ref{prop:special values of H and K 2}, 
Corollary\,\ref{prop:special values of deriv of H and K}, 
Proposition\,\ref{prop:difference eq of H and K}, 
\ref{prop:difference eq 2 of H and K}, 
\ref{prop:multiplication of H and K} and 
Corollary\,\ref{prop:deriv of H and K at s} 
respectively.

For (\ref{eq:Raabe formulas}), 
\begin{align}
n\int_{0}^{l}\log{\mathscr{S}_{r,n}(z+t)}\,dt
&=\int_{0}^{l}\frac{d}{dt}\log{\mathscr{S}_{r,n+1}(z+t)}\,dt \nonumber \\
&=\log{\mathscr{S}_{r,n+1}(z+l)}-\log{\mathscr{S}_{r,n+1}(z)} \nonumber \\
&=\sum_{k=1}^{r-1}(-l)^{k}\binom{r-1}{k}\log{\mathscr{S}_{r-k,n+1}(z)}. \nonumber
\end{align}
The first equality follows from (\ref{eq:deriv of mathscr S 2}) and the third one follows from (\ref{eq:difference eq of mathcal S and mathscr S}). 
The proof of (\ref{eq:Raabe formulas 2}) can be similarly.

Furthermore, we also obtain the relations between our generalized multiple sine functions and known multiple sine function of \cite{KK}, \cite{KW}.  
\begin{prop}
\label{prop:relation between our fnc and known ones}
{\rm{(1)}}\,
\begin{equation}
\mathscr{S}_{r,1}(z)=\exp\left(-\frac{(r-1)!}{(2{\pi}i)^{r-1}}\zeta(r)+(-1)^{r-1}{\pi}i\frac{z^{r}}{r}\right)\mathscr{S}_{r}(z)^{(-1)^{r-1}}.
\end{equation}
{\rm{(2)}}\,
\begin{equation}
\mathcal{S}_{r,n}(z)=\exp\left((-1)^{r}\pi{i}\frac{(n-1)!}{(r+n-1)!}B_{r,r+n-1}(z)\right){S_{r,n}(z)}.
\end{equation}
Here, $B_{r,k}(z)$ is the multiple Bernoulli polynomials defined by a generating function as
$$
\frac{t^{r}e^{zt}}{(e^{t}-1)^{r}}
=\sum_{k=0}^{\infty}B_{r,k}(z)\frac{t^{k}}{k!}. 
$$

\end{prop}
\begin{proof}
{\rm{(1)}}\,
The logarithmic derivative of $\mathscr{S}_{r,1}(z)$ is 
$$
\frac{\mathscr{S}_{r,1}'(z)}{\mathscr{S}_{r,1}(z)}
=(-z)^{r-1}(\pi{i}+\pi{\cot{\pi z}})
=(-1)^{r-1}\pi{i}z^{r-1}+(-1)^{r-1}\frac{\mathscr{S}_{r}'(z)}{\mathscr{S}_{r}(z)}.
$$
The first equality follows from (\ref{eq:deriv of H s0}) and the second one from the integral expression of $\mathscr{S}_{r}(z)$
$$
\mathscr{S}_{r}(z)=\exp{\left(\int_{0}^{z}\pi t^{r-1}\cot{({\pi}t)}\,dt\right)}.
$$
Thus, there exists some constant $C$ such that 
$$
\mathscr{S}_{r,1}(z)=C\exp\left((-1)^{r-1}{\pi}i\frac{z^{r}}{r}\right)\mathscr{S}_{r}(z)^{(-1)^{r-1}}.
$$
In addition, by putting $z=0$, 
$$
C=\mathscr{S}_{r,1}(0).
$$
{\rm{(2)}}\,From the $\zeta_{r}(s,z)$ expression of $K_{r}(s,z)$ (\ref{eq:zeta expression of K}), 
\begin{align}
\mathcal{S}_{r,n}(z)&=\exp\left(-\frac{\partial K_{r}}{\partial s}(1-n,z)\right) \nonumber \\
&=\exp{\left((-1)^{r+n-1}{\pi}i\zeta_{r}(1-n,r-z)-\frac{\partial \zeta_{r}}{\partial s}(1-n,z)+(-1)^{r+n-1}\frac{\partial \zeta_{r}}{\partial s}(1-n,r-z)\right)} \nonumber \\
&=\exp\left((-1)^{r}\pi{i}\frac{(n-1)!}{(r+n-1)!}B_{r,r+n-1}(z)\right)S_{r,n}(z). \nonumber
\end{align}
The third equality follows from special values of $\zeta_{r}(1-n,X)$
$$
\zeta_{r}(1-n,X)=(-1)^{r}\frac{(n-1)!}{(n+r-1)!}B_{r,r+n-1}(X)
$$
and some formula of multiple Bernoulli polynomial
$$
B_{r,k}(r-X)=(-1)^{k}B_{r,k}(X),
$$
(see \cite{B}).
\end{proof}



\bibliographystyle{amsplain}

\noindent Institute of Mathematics for Industry, Kyushu University\\
744, Motooka, Nishi-ku, Fukuoka, 819-0395, JAPAN.\\
E-mail: g-shibukawa@math.kyushu-u.ac.jp
\end{document}